\documentclass[reqno,12pt]{amsart}
\usepackage{latexsym}
\usepackage{amsxtra}
\usepackage{verbatim}
\usepackage{color}
\usepackage{amsthm,amsmath,amsfonts,amssymb,amscd,mathrsfs,graphics}
\usepackage{paralist}
\usepackage{amssymb}
\usepackage{times,float}
\usepackage{appendix}
\usepackage{txfonts}
\usepackage{hyperref,supertabular}
\usepackage{ifpdf}
\usepackage{enumerate}
\usepackage{fancyhdr}
\usepackage[top=1.5in, bottom=1.5in, left=1.5in, right=1.5in]{geometry}

\allowdisplaybreaks

\newcommand{\beq}{\begin{equation}}
\newcommand{\eeq}{\end{equation}}
\newcommand{\beqa}{\begin{eqnarray}}
\newcommand{\eeqa}{\end{eqnarray}}
\newcommand{\beaa}{\begin{eqnarray*}}
\newcommand{\ben}{\begin{eqnarray*}}
\newcommand{\eaa}{\end{eqnarray*}}
\newcommand{\een}{\end{eqnarray*}}

\newcommand \nc {\newcommand}

\newtheorem{theorem}{Theorem}[section]
\newtheorem{lemma}[theorem]{Lemma}

\newtheorem{definition}[theorem]{Definition}

\nc \thref{Theorem \ref}
\nc \leref{Lemma \ref}
\nc \prref{Proposition \ref}
\nc \coref{Corollary \ref}
\nc \deref{Definition \ref}
\nc \exref{Example \ref}
\nc \reref{Remark \ref}

\newcommand{\A}{\mathcal{A}}
\newcommand{\B}{\mathcal{B}}
\newcommand{\C}{\mathbb{C}}
\newcommand{\D}{\mathcal{D}}

\newcommand{\F}{\mathcal{F}}
\renewcommand{\H}{\mathcal{H}}

\newcommand{\M}{\mathcal{M}}

\renewcommand{\O}{\mathcal{O}}

\newcommand{\T}{\mathcal{T}}

\newcommand{\Z}{\mathbb{Z}}

\newcommand{\f}{\mathbf{f}}

\newcommand{\q}{\mathbf{q}}
\renewcommand{\t}{\mathbf{t}}

\def\dim{\mathop{\rm dim}\nolimits}

\def\d{\partial}

\def\({\left(}
\def\){\right)}
\def\[{\left[}
\def\]{\right]}
\def\<{\left\langle}
\def\>{\right\rangle}

\def\lieh{{\mathfrak{h}}}

\def\one{{\bf 1}}

\def\D{{\mathcal D}}
\def\la{\lambda}
\def\ge{\epsilon}
\def\al{\alpha}
\def\de{\delta}
\def\be{\beta}
\def\Si{\Sigma}

\newcommand{\leftexp}[2]{{\vphantom{#2}}^{#1}{#2}}

\title[Analyticity of the ancestors in singularity theory]
{Analyticity of  the total ancestor potential in singularity theory}
\author{Todor Milanov}
\address{Kavli IPMU (WPI) \\ The University of Tokyo \\ Kashiwa \\ Chiba 277-8583 \\ Japan}
\email{todor.milanov@ipmu.jp}
\thanks{{\em 2000 Math. Subj. Class.} 14D05, 14N35, 17B69}
\thanks{
{\em Key words and phrases:} period integrals, Frobenius structure,
Goromov--Witten invariants, vertex operators}

\begin{document}

\begin{abstract}
K. Saito's theory of primitive forms gives a natural semi-simple
Frobenius manifold structure on the space of miniversal deformations of an
isolated singularity. On the other hand, Givental introduced the notion of a total
ancestor potential for every semi-simple point of a Frobenius manifold and
conjectured that in the settings of singularity theory his definition
extends analytically to non-semisimple points as well. In this
paper we prove Givental's conjecture by using the Eynard--Orantin
recursion.   
\end{abstract}
\maketitle
\tableofcontents
\addtocontents{toc}{\protect\setcounter{tocdepth}{1}}

\section{Introduction}

The Gromov--Witten invariants of a compact algebraic manifold $V$ are by
definition a virtual count of holomorphic maps from a Riemann surface
to $V$ satisfying various incidents constraints. Although the
rigorous definition of the Gromov--Witten invariants is very
complicated, when it comes to computations, quite a bit of techniques
were developed. One of the most exciting achievements is due to Givental
who conjectured that under some technical conditions, which amount to
saying that $V$ has sufficiently many rational curves, we can
reconstruct the higher genus invariants in terms of genus 0 and the
higher genus Gromov--Witten invariants of the point. Givental's
conjecture was proved recently by Teleman \cite{Te} and its impact on
other areas of mathematics, such as integrable systems and
the theory of quasi-modular forms is a subject of an ongoing
investigation (see \cite{BPS}, \cite{MR}). 

The higher genus reconstruction formalism of Givental (see \cite{G1}
or Section \ref{sec4} bellow)  is most
naturally formulated in the abstract settings of the so called  {\em semi-simple
Frobenius manifolds} (see \cite{Du} for some background on Frobenius
manifolds). In the case of Gromov--Witten theory, the Frobenius
structure is given on the vector space $H^*(V;\C)$ and it is induced
from the quantum cup product. More precisely, Givental defined the
total ancestor potential of a semi-simple Frobenius manifold which in
the case of Gromov--Witten theory coincides with a generating
function of the so called {\em ancestor} Gromov--Witten invariants
(see \cite{G2}). 

In this paper we study the total ancestor potential of the semi-simple
Frobenius manifold arising in singularity theory. 
Let $f\in \O_{\C^{2l+1},0}$ be the germ of a holomorphic function with an
isolated critical point at $0$, i.e., the local algebra
$H:=\O_{\C^{2l+1},0}/(f_{x_0},\dots, f_{x_{2l}})$ is a finite
dimensional vector space (over $\C$). The dimension is called {\em
  multiplicity} of the critical point and it will be denoted by $N.$
We fix a miniversal deformation $F(t,x),$ $t\in B$ and a primitive
form $\omega$ in the sense of K. Saito \cite{Sa, MS}, so that $B$
inherits a Frobenius structure (see \cite{He,SaT}). 
Let $B_{\rm ss}$ be the set of points $t_0\in B$, such that the critical
values $u_1(t),\dots,u_N(t)$ of $F(t,\cdot)$ form a coordinate system
for $t$ in a neighborhood of $t_0.$ In such coordinates the product
and the residue pairing assume a diagonal form which means that the
corresponding Frobenius algebra is semi-simple. 
Let $\t=\{t_{k,i}\}_{k=0,1,\dots}^{i=1,\dots,N}$ be a sequence of
formal variables. 
For every $t\in B_{\rm ss}$ we denote by $\A_t(\hbar;\t)$ the total ancestor potential
of the Frobenius structure (c.f. Section \ref{asop}). It is a formal
power series  in $\t$ with coefficients formal Laurent series in
$\hbar$, whose coefficients are analytic functions in $t\in B_{ss}$. 
\begin{theorem}\label{t1}
The total ancestor potential of an isolated singularity corresponding
to the Frobenius structure of an arbitrary primitive form extends
analytically to all $t\in B$. 
\end{theorem}
The fact that
we have higher genus reconstruction at non-semisimple points looks
quite attractive on its own and it deserves a further
investigation. In particular, it will be interesting to find a
generalization of Givental's formula at various non-semisimple points and
see if similar formulas occur in Gromov--Witten theory as
well. Finally, let us point out that Theorem \ref{t1} is
very important for the Landau--Ginzburg/Calabi--Yau
correspondence (c.f. \cite{MR, MRS}) where it is necessary to restrict
the total ancestor potential to {\em marginal deformations} only and
the latter are always non-semisimple. 

The proof of Theorem \ref{t1} is based on the local Eynard--Orantin
recursion (see \cite{BOSS} and \cite{M}). We follow the approach in
\cite{M}. The main advantage of the recursion is that it gives a
reconstruction which does not make use of the higher genus theory of
the point, but it depends only on the Frobenius structure! Following
an idea of Bouchard--Eynard (see \cite{BE}) we
prove that the local recursion, which apriori is defined only for
$t\in B_{\rm ss}$, extends to generic points $t\in B\setminus{B_{\rm
    ss}}.$ Let us point out that at this point we use the fact that for a generic
$t\in B\setminus{B_{\rm ss}}$ the function $F(t,\cdot)$ has a singularity of
type $A_2$. From here one proves easily by induction that $A_t(\hbar;\q)$ extends
analytically for generic $t\in B\setminus{B_{\rm ss}}$ provided some
initial set of correlators of genus 0 and genus 1 are analytic. While
the analyticity of the genus 0 correlators is easy to verify, the
analyticity of the genus 1 ones is much more involved. However the
computation was already done by C. Hertling (see \cite{He}, Theorem
14.6). Therefore, to complete the proof of Theorem \ref{t1}, it remains
only to recall the Hartogue's extension theorem.

\section{Frobenius structures in singularity theory}\label{sec:sing}

Let us first recall some of the basic settings in singularity
theory. For more details we refer the reader to the excellent book \cite{AGV}. 
Let $f\colon(\C^{2l+1},0)\rightarrow (\C,0)$ be the germ of a holomorphic function with an isolated critical point of multiplicity $N$. Denote by 
\begin{equation*}
H = \C[[x_0,\ldots,x_{2l}]]/(\d_{x_0}f,\ldots,\d_{x_{2l}}f)
\end{equation*}
the \emph{local algebra} of the critical point; then $\dim H=N$. 

\begin{definition}\label{dmvdef}
A \emph{miniversal deformation} of $f$ is a germ of a holomorphic function $F\colon(\C^N\times \C^{2l+1},0)\to (\C,0)$ satisfying the following two properties:
\begin{enumerate}
\item[(1)]
$F$ is a deformation of $f$, i.e., $F(0,x)=f(x)$.
\item[(2)]
The partial derivatives $\d F/\d t^i$ $(1\leq i\leq N)$ project to a basis in the local algebra 
$$
\O_{\C^N,0}[[x_0,\dots,x_{2l}]]/\langle \d_{x_0}F,\dots,\d_{x_{2l}}F\rangle.
$$
\end{enumerate}
Here we denote by $t=(t^1,\dots,t^N)$ and $x=(x_0,\dots,x_{2l})$ the standard coordinates on $\C^N$ and $\C^{2l+1}$ respectively, and $\O_{\C^N,0}$ is the algebra of germs at $0$ of holomorphic functions on $\C^N.$
\end{definition}

We fix a representative of the holomorphic germ $F$, which we denote again by $F$, with a domain $X$ constructed as follows. Let 
\begin{equation*}
B_\rho^{2l+1}\subset \C^{2l+1} \,, \qquad 
B=B_\eta^N\subset \C^N \,, \qquad 
B_\delta^1\subset \C 
\end{equation*}
be balls with centers at $0$ and radii $\rho,\eta$, and $\delta$, respectively.
We set 
\begin{equation*}
S=B\times B_\delta^1 \subset\C^N\times\C \,, \quad 
X=(B\times B_\rho^{2l+1})\cap \phi^{-1}(S) \subset\C^N\times\C^{2l+1} \,,
\end{equation*}
where
\ben
\phi\colon B\times B_\rho^{2l+1}\to B\times\C \,,
\qquad (t,x)\mapsto (t,F(t,x)) \,.
\een 
This map induces a map $\phi\colon X\to S$ and we denote by $X_s$ or $X_{t,\la}$ the fiber 
\ben
X_s = X_{t,\la} = \{(t,x)\in X \,|\, F(t,x)=\la\} \,,\qquad s=(t,\la)\in S.
\een  
The number $\rho$ is chosen so small that for all $r$, $0<r\leq \rho$, the fiber $X_{0,0}$ intersects transversely the boundary $\d B_r^{2l+1}$ of the ball with radius $r$. Then we choose the numbers $\eta$ and $\delta$ small enough so that for all $s\in S$ the fiber $X_s$ intersects transversely the boundary $\d B_\rho^{2l+1}.$ Finally, we can assume without loss of generality that the critical values of $F$ are contained in a disk $B_{\delta_0}^1$ with radius $\delta_0<1<\delta$.

Let $\Si$ be the {\em discriminant} of the map $\phi$, i.e., the set
of all points $s\in S$ such that the fiber $X_s$ is singular. Put 
\begin{equation*}
S'=S\setminus{\Si} \subset\C^N\times\C \,, \qquad 
X'=\phi^{-1}(S') \subset X \subset\C^N\times\C^{2l+1} \,.
\end{equation*}
Then the map $\phi\colon X'\to S'$ is a smooth fibration, called 
the \emph{Milnor fibration}. In particular, all smooth fibers are diffeomorphic to $X_{0,1}$.
The middle homology group of the smooth fiber, equipped with the bilinear form
$(\cdot|\cdot)$ equal to $(-1)^l$ times the intersection form, 
is known as the \emph{Milnor lattice} $Q=H_{2l}(X_{0,1};\Z)$. 
For a generic point $s\in\Si$, the singularity of the fiber $X_s$
is Morse. Thus, every choice of a path from $(0,1)$ to $s$ avoiding $\Si$
leads to a group homomorphism $Q \to H_{2l}(X_s;\Z)$. The kernel of this
homomorphism is a free $\Z$-module of rank $1$. A generator  
$\al\in Q$ of the kernel is called a \emph{vanishing cycle} if 
$(\al|\al) = 2$. 

\subsection{Frobenius structure}\label{sec:frobenius}
Let $\T_B$ be the sheaf of holomorphic vector fields on $B$. Condition (2) in \deref{dmvdef} implies that the map 
$$
\d/\d{t^i}\mapsto \d F/\d t^i \mod \langle \d_{x_0} F,\dots,\d_{x_{2l}}F\rangle \qquad (1\leq i\leq N)
$$ 
induces an isomorphism between $\T_B$ and $p_*\O_C$, where $p\colon X\to B$ is the natural projection $(t,x)\mapsto t$ and 
\ben
\O_C:=\O_X/\langle \d_{x_0} F,\dots,\d_{x_{2l}}F\rangle
\een
is the structure sheaf of the critical set of $F$. In particular, since $\O_C$ is an algebra, the sheaf $\T_B$ is equipped with an associative commutative multiplication, which will be denoted by $\bullet.$ It induces a product $\bullet_t$ on the tangent space of every point $t\in B$. The class of the function $F$ in $\O_C$ defines a vector field $E\in \T_B$, called the {\em Euler vector field}. 

Given a holomorphic volume form $\omega$ on $(\C^{2l+1},0)$, possibly
depending on $t\in B$, we can equip $p_*\O_C$ with the so-called
\emph{residue pairing}
\ben
(\psi_1(t,x),\psi_2(t,x)) :=
\Big(\frac{1}{2\pi i}\Big)^{2l+1}\int_{\Gamma_\epsilon} 
\frac{\psi_1(t,y)\ \psi_2(t,y)}
{\d_{y_0} F \cdots \d_{y_{2l}} F } \,\omega\,,
\een
where $y=(y_0,\dots,y_{2l})$ is a $\omega$-unimodular coordinate system
(i.e. $\omega=dy_0\wedge \cdots \wedge dy_{2l}$) and the integration cycle $\Gamma_\epsilon$ is supported on 
$|\d_{y_0} F|= \cdots =|\d_{y_{2l}}F|=\epsilon$. 
Using that $\T_B\cong p_*\O_C$, we get a non-degenerate complex
bilinear form $(\ ,\ )$ on $\T_B$, which we still call residue
pairing.    

For $t\in B$ and $z\in\C^*$, let
$\B_{t,z}$ be a semi-infinite cycle in $\C^{2l+1}$ of the following type:
\ben
\B_{t,z}\in \lim_{\rho \to \infty} \, H_{2l+1}(\C^{2l+1},\{ 
  \mathrm{Re}\, z^{-1} F(t,x)<-\rho\} ;\C) \cong \C^N \,.
\een
The above homology groups form a vector bundle on $B\times \C^*$ equipped
naturally with a Gauss--Manin connection, and $\B=\B_{t,z}$ may be viewed as a
flat section. According to K.\ Saito's theory of {\em primitive forms} \cite{He,Sa,MS}
there exists a form $\omega$, called primitive, such that the oscillatory
integrals ($d^B$ is the de Rham differential on $B$)
\ben
J_\B(t,z):= (2\pi z)^{-l-\frac{1}{2}}\ (zd^B)\, 
\int_{\B_{t,z}} e^{z^{-1}F(t,x)}\omega \in \T_B^*
\een
are horizontal sections for the following connection: 
\beqa\label{frob_eq3}
\nabla_{\d/\d t^i} & = &  \nabla^{\rm L.C.}_{\d/\d t^i} - z^{-1}(\d_{t^i} \bullet_t),
\qquad 1\leq i\leq N \\
\label{frob_eq4}
\nabla_{\d/\d z} & = &  \d_z - z^{-1} \theta + z^{-2} E\bullet_t \,.
\eeqa
Here $\nabla^{\rm L.C.}$ is the Levi--Civita connection associated with the residue pairing and
\ben
\theta:=\nabla^{\rm L.C.}E-\Big(1-\frac{d}{2}\Big){\rm Id},
\een 
where $d$ is some complex number. 
In particular, this means that the residue pairing and the multiplication $\bullet$ form a {\em Frobenius structure} on $B$
of conformal dimension $d$ with identity $1$ and Euler vector field $E$. For the definition of a Frobenius structure we refer to \cite{Du} .

Assume that a primitive form $\omega$ is chosen. Note that the flatness of the Gauss--Manin connection implies that the residue pairing is flat. Denote by $(\tau_1,\dots, \tau_N)$ a coordinate system on $B$ that is flat with respect to the residue pairing, and write $\partial_i$ for the vector field $\partial/\partial{\tau_i}$. We can further modify the flat coordinate system so that the Euler field is the sum of a constant and linear fields: 
\ben
E=\sum_{i=1}^N (1-d_i) \tau_i \partial_{i} + \sum_{i=1}^N \rho_i \partial_i \,.
\een
The constant part represents the class of $f$ in $H$, and the spectrum
of degrees $d_1,\dots, d_N$ ranges from $0$ to $d.$ 
Note that in the flat coordinates $\tau_i$ the operator $\theta$ (called sometimes the \emph{Hodge grading operator}) assumes diagonal form:
\ben
\theta(\d_i) = \Bigl(\frac{d}{2}-d_i\Bigr) \d_i \,, \qquad\quad 1\leq i\leq N \,.
\een
Finally, let us trivialize the tangent and the cotangent bundle. We have the following identifications:
\ben
T^*B\cong TB\cong B\times T_0B\cong B\times H,
\een
where $H$ is the Jacobi algebra of $f$, the first isomorphism is given by the residue pairing, the
second by the Levi--Cevita connection of the flat residue pairing, and the last one is the
Kodaira--Spencer isomorphism 
\beq\label{KS}
T_0B \cong H,\quad \d/\d t_i\mapsto \left. \d_{t_i}F\right|_{t=0} \ {\rm mod}\ 
 ( f_{x_0},\dots,f_{x_{2l}}).
\eeq
Let $v_i\in H$ be the images of the flat vector fields $\d_i$
via the Kodaira--Spencer isomorphism \eqref{KS}. We assume that
$v_N=1$ is the unity of the algebra $H$. 

\subsection{Period integrals}\label{sec:periods}
Given a middle homology class $\al\in H_{2l}(X_{0,1};\C)$, we denote by $\al_{t,\la}$ its parallel transport to 
the Milnor fiber $X_{t,\la}$. Let $d^{-1}\omega$ be any $2l$-form whose 
differential is $\omega$. We can integrate $d^{-1}\omega$ over $\al_{t,\la}$
and obtain multivalued functions of $\la$ and $t$ 
ramified around the discriminant in $S$ (over which 
the Milnor fibers become singular). 
To $\al\in H_{2l}(X_{0,1};\C)$, we associate the {\em 
period vectors} $I^{(k)}_\al(t,\la)\in H\ (k\in \Z)$ defined by
\beq\label{periods} 
(I^{(k)}_\al(t,\la), v_i):= 
-(2\pi)^{-l} \d_\la^{l+k} \d_i \int_{\al_{t,\la}} d^{-1}\omega \,,
\qquad 1\leq i\leq N \,.
\eeq 
Note that this definition is consistent with the operation of stabilization of
singularities. Namely, adding the squares of two new variables does not change
the right-hand side, since it is offset by an extra differentiation 
$(2\pi)^{-1}\partial_{\la}$. In particular, this defines the period
vector for a negative value of $k\geq -l$ with $l$ as large as one wishes.
Note that, by definition, we have 
\ben
\d_\la I^{(k)}_\al(t,\la) = I^{(k+1)}_\al(t,\la) \,, \qquad k\in\Z\,.
\een
The following lemma is a consequence of the definition
of a primitive form.
\begin{lemma}\label{lem:periods}
The period vectors \eqref{periods} satisfy the differential 
equations
\begin{align}\label{periods:de1}
\d_i I_\al^{(k)} &= -v_i\bullet_t(\d_\la I_\al^{(k)})\,, \qquad\quad 1\leq i\leq N \,, 
\\ \label{periods:de2}
(\la-E\bullet_t) \d_{\la}I_\al^{(k)} &= \Bigl(\theta-k-\frac12\Bigr) I_\al^{(k)} \,.
\end{align}
\end{lemma}
The connection corresponding to the differential equations
\eqref{periods:de1}--\eqref{periods:de2} is a Laplace transform of the
connection \eqref{frob_eq3}--\eqref{frob_eq4}. In particular, since the
oscillatory integrals are related to the period vectors via the 
Laplace transform, Lemma \ref{lem:periods} follows from the fact that
the oscillator integrals are horizontal sections for the connection
\eqref{frob_eq3}--\eqref{frob_eq4}. 

Using equation \eqref{periods:de2},
we analytically extend the period vectors to all 
$|\la|>\delta$. It follows from \eqref{periods:de1} that the period vectors 
have the symmetry
\beq\label{tinv}
I^{(k)}_\al(t,\la)\ = \ I^{(k)}_\al(t-\la\one,0) \,,
\eeq  
where $t \mapsto t-\la\one $ denotes the time-$\la$ translation 
in the direction of the flat vector field $\one $ obtained from $1\in H$. 
(The latter represents identity elements for all the products
$\bullet_t$.)

Let $t\in B_{ss}$ be a semi-simple point; then the period vector
$I^{(0)}_\al(t,\la)$ could have singularities only at the critical
values $u_i(t)$. Moreover, using equation \eqref{frob_eq4} it is easy
to see that the order of the pole at a given singular point
$\la=u_i(t)$ is at most $\frac{1}{2}$. A simple corollary of this
observation, which will be used repeatedly, is that if the cycle $\al$ is invariant with respect to
the local monodromy around $\la=u_i(t)$; then the corresponding period
vectors $I^{(n)}_\al(t,\la)$ must be analytic in a neighborhood of
$\la=u_i(t)$.

\subsection{Stationary phase asymptotic}\label{sec:asymptotic}

Let $u_i(t)$ ($1\leq i\leq N$) be the critical values of $F(t,\cdot)$. For a generic $t$, they form a local coordinate system on $B$ in which the Frobenius multiplication and the residue pairing are diagonal. Namely,
\ben
\d/\d u_i \, \bullet_t\, \d/\d u_j = \delta_{ij}\d/\d u_j \,,\quad 
\(\d/\d u_i,\d/\d u_j \) = \delta_{ij}/\Delta_i \,,
\een
where $\Delta_i$ is the Hessian of $F$ with respect to the volume form $\omega$ at the critical point corresponding to the critical value $u_i.$ 
Therefore, the Frobenius structure is \emph{semi-simple}.
We denote by $\Psi_t$ the following linear isomorphism
\ben
\Psi_t\colon \C^N\rightarrow T_t B \,,\qquad 
e_i\mapsto \sqrt{\Delta_i}\d/\d u_i \,,
\een
where $\{e_1,\dots,e_N\}$ is the standard basis for $\C^N$.

Let $U_t$ be the diagonal matrix with entries $u_1(t),\ldots, u_N(t)$. 
According to Givental \cite{G1}, the system of differential equations (cf.\ \eqref{frob_eq3}, \eqref{frob_eq4})
\begin{align}\label{de_1}
z\d_i J(t,z) &= v_i\bullet_t J(t,z) \,,\qquad\quad 1\leq i\leq N \,,
\\
\label{de_2}
z\d_z J(t,z) &= (\theta-z^{-1} E\bullet_t) J(t,z)
\end{align}
has a unique formal asymptotic solution of the form $\Psi_t R_t(z) e^{U_t/z}$, 
where 
\ben
R_t(z)=1+R_1(t)z+R_2(t)z^2+\cdots \,,
\een
and $R_k(t)$ are linear operators on $\C^N$ uniquely determined from the differential 
equations \eqref{de_1} and \eqref{de_2}.

We will make use of the following formal series
\beq\label{falpha}
\f_\al(t,\la;z) = \sum_{k\in \Z} I^{(k)}_\al(t,\la) \, (-z)^k \, ,
\eeq
and
\beq\label{phi-alpha}
\phi_\al(t,\la;z) = \sum_{k\in \Z} I^{(k+1)}_\al(t,\la) \,d\la\,  (-z)^k \, .
\eeq
Note that for $A_1$-singularity $F(t,x)=x^2/2+t$ we have
$u:=u_1(t)=t.$ Up to a sign there is a unique vanishing cycle. The
corresponding series \eqref{falpha}  and \eqref{phi-alpha} will be
denoted simply by $\f_{A_1}(t,\la;z)$ and $\phi_{A_1}(t,\la;z).$ The
period vectors can be computed explicitly and they are given by the
following formulas:
\ben
\begin{aligned}
I^{(k)}_{A_1}(u,\la) & = (-1)^k\, \frac{(2k-1)!!}{2^{k-1/2}}\,
(\la-u)^{-k-1/2},\quad k\geq 0 \\
I^{(-k-1)}_{A_1}(u,\la) & = 2\, \frac{2^{k+1/2}}{(2k+1)!!}\,
(\la-u)^{k+1/2}, \quad k\geq 0.
\end{aligned}
\een
The key lemma (see \cite{G3}) is the following.
\begin{lemma}\label{vanishing_a1}
Let\/ $t\in B$ be generic and\/ $\be$ be a vanishing cycle vanishing over the point\/ $(t,u_i(t))\in \Si$. Then for all\/ $\la$ near\/ $u_i:=u_i(t)$, we have
\ben
\f_{\be}(t,\la;z) = \Psi_t R_t(z)\,  e_i\,  \f_{A_1}(u_i,\la;z)\,.
\een
\end{lemma}

\section{Symplectic loop space formalism}\label{sec4}

The goal of this section is to introduce Givental's quantization
formalism (see \cite{G2}) and use it to define the higher genus potentials in
singularity theory.

\subsection{Symplectic structure and quantization}\label{ssympl}

The space $\H:=H(\!(z^{-1})\!)$ of formal Laurent series in $z^{-1}$ with
coefficients in $H$ is equipped with the following \emph{symplectic form}: 
\ben
\Omega(\phi_1,\phi_2):={\rm Res}_z \(\phi_1(-z),\phi_2(z)\) \,,
\qquad \phi_1,\phi_2\in\H \,,
\een 
where, as before, $(,)$ denotes the residue pairing on $H$
and the formal residue ${\rm Res}_z$ gives the coefficient in front of $z^{-1}$.

Let $\{v_i\}_{i=1}^{ N}$ and $\{v^i\}_{i=1}^N$ be dual bases of $H$ with respect to the residue pairing.
Then
\ben
\Omega(v^i(-z)^{-k-1}, v_j z^l) = \de_{ij} \de_{kl} \,.
\een
Hence, a Darboux coordinate system is provided by the linear functions $q_k^i$, $p_{k,i}$ on $\H$ given by:
\ben
q_k^i = \Omega(v^i(-z)^{-k-1}, \cdot) \,, \qquad
p_{k,i} = \Omega(\cdot, v_i z^k) \,.
\een
In other words,
\ben
\phi(z) = \sum_{k=0}^\infty \sum_{i=1}^N q_k^i(\phi) v_i z^k +  
\sum_{k=0}^\infty \sum_{i=1}^N p_{k,i}(\phi) v^i(-z)^{-k-1} \,,
\qquad \phi\in\H \,.
\een  
The first of the above sums will be denoted $\phi^+(z)$ and the second
$\phi^-(z)$.

The \emph{quantization} of linear functions on $\H$ is given by the rules:
\ben
\widehat q_k^i = \hbar^{-1/2} q_k^i \,, \qquad
\widehat p_{k,i} = \hbar^{1/2} \frac{\d}{\d q_k^i} \,.
\een
Here and further, $\hbar$ is a formal variable. We will denote by $\C_\hbar$ the field
$\C(\!(\hbar^{1/2})\!)$.

Every $\phi(z)\in\H$ gives rise to the linear function $\Omega(\phi,\cdot)$ on $\H$,
so we can define the quantization $\widehat\phi$. Explicitly,
\beq\label{phihat}
\widehat\phi = -\hbar^{1/2} \sum_{k=0}^\infty \sum_{i=1}^N q_k^i(\phi) \frac{\d}{\d q_k^i}
+ \hbar^{-1/2} \sum_{k=0}^\infty \sum_{i=1}^N p_{k,i}(\phi) q_k^i \,.
\eeq
The above formula makes sense also for $\phi(z)\in H[[z,z^{-1}]]$ if we interpret $\widehat\phi$
as a formal differential operator in the variables $q_k^i$ with coefficients in $\C_\hbar$.

\begin{lemma}\label{lphihat}
For all\/ $\phi_1,\phi_2\in\H$, we have\/ $[\widehat\phi_1, \widehat\phi_2] = \Omega(\phi_1,\phi_2)$.
\end{lemma}
\begin{proof}
It is enough to check this for the basis vectors $v^i(-z)^{-k-1}$, $v_i z^k$, in which case it is true by definition.
\end{proof}
It is known that the operator series
$
\mathcal{R}_t(z):=\Psi_t R_t(z)\Psi_t^{-1}
$
is a symplectic transformation. Moreover, it has the form $e^{A(z)},$ where $A(z)$ is an infinitesimal symplectic transformation. 
A linear operator $A(z)$ on $\H:=H(\!(z^{-1})\!)$ is infinitesimal symplectic if and only if the map $\phi\in \H \mapsto A\phi\in \H$ is a Hamiltonian vector field with a Hamiltonian given by the quadratic function $h_A(\phi) = \frac{1}{2}\Omega(A\phi,\phi)$. 
By definition, the \emph{quantization} of $e^{A(z)}$ is given by the differential operator $e^{\widehat{h}_A},$ where the quadratic Hamiltonians are quantized according to the following rules:
\ben
(p_{k,i}p_{l,j})\sphat = \hbar\frac{\d^2}{\d q_k^i\d q_l^j} \,,\quad 
(p_{k,i}q_l^j)\sphat = (q_l^jp_{k,i})\sphat = q_l^j\frac{\d}{\d q_k^i} \,,\quad
(q_k^iq_l^j)\sphat = \frac1{\hbar} q_k^iq_l^j \, .
\een     

\subsection{The total ancestor potential }\label{asop}
Let us make the following convention. Given a vector 
\ben
\q(z) = \sum_{k=0}^\infty q_k z^k \in H[z] \,, \qquad
q_k=\sum_{i=1}^N q_k^iv_i \in H \,,
\een
its coefficients give rise to a vector sequence  $q_0,q_1,\dots$.
By definition, a {\em formal  function} on $H[z]$,
defined in the formal neighborhood of a given point $c(z)\in H[z]$, is a formal power
series in $q_0-c_0, q_1-c_1,\dots$. Note that every operator acting on
$H[z]$ continuously in the appropriate formal sense induces an
operator acting on formal functions.

The \emph{Witten--Kontsevich tau-function} is the following generating series:
\beq\label{D:pt}
\D_{\rm pt}(\hbar;Q(z))=\exp\Big( \sum_{g,n}\frac{1}{n!}\hbar^{g-1}\int_{\overline{\M}_{g,n}}\prod_{i=1}^n (Q(\psi_i)+\psi_i)\Big),
\eeq
where $Q_0,Q_1,\ldots$ are formal variables, and $\psi_i$ ($1\leq
i\leq n$) are the first Chern classes of the cotangent line bundles on
$\overline{\M}_{g,n}$ (see \cite{W1,Ko1}).
It is interpreted as a formal
function of $Q(z)=\sum_{k=0}^\infty Q_k z^k\in \C[z]$, defined in the
formal neighborhood of $-z$. In other words, $\D_{\rm pt}$ is a formal power series
in $Q_0,Q_1+1,Q_2,Q_3,\dots$ with coefficients in $\C(\!(\hbar)\!)$.

Let $t\in B$ be a {\em semi-simple} point, so that the critical values
$u_i(t)$ ($1\leq i\leq N$) of $F(t,\cdot)$ form a coordinate
system. Recall also the flat coordinates
$\tau=(\tau_1(t),\dots,\tau_N(t))$ of $t$. The {\em total ancestor
  potential} of the singularity is defined as follows
\beq\label{ancestor}
\A_t(\hbar; \q(z)) = \widehat{\mathcal{R}}_t\ \prod_{i=1}^N\, \D_{\rm pt}(\hbar\Delta_i;\leftexp{i}{\q}(z)) \in \C_\hbar[[q_0,q_1+\one,q_2\dots]],
\eeq
where 
$
\mathcal{R}_t(z):=\Psi_t R_t(z)\Psi_t^{-1}
$
and 
\ben
\leftexp{i}{\q}(z) = \sum_{k=0}^\infty \sum_{a=1}^N \ \frac{\d u_i}{\d
  \tau_a}\, q_k^a \,z^k \,.
\een
It will be convenient also to consider another set $\t=\{t_k^i\}$ of formal variables
related to $\q$ via the so called {\em dilaton shift:}
\beq\label{dilaton}
t_k^i=
\begin{cases}
q_k^i & \mbox{ if } k\neq 1, i\neq N \\
q_1^N+1 & \mbox{ otherwise,}
\end{cases}
\eeq
where recall that $v_N=1\in H$ is the unit for the Frobenius
multiplication. 

\section{The ancestors for generic non-semisimple points}

Let $t_0\in B$ be a point, such that the function $F(t_0,\cdot
):B_\rho^{2l+1}\to \C$ has $N-2$ Morse critical points $\xi_i^0(1\leq
i\leq N-2)$ and a critical
point $\xi_{N-1}^0$ of type $A_2$, i.e., we can choose local coordinate system
$y=(y_0,y_1,\dots,y_{2l})$ centered at the critical point $\xi_{N-1}^0$ such that 
\ben
F(t_0,y)=u_{N-1}^0+y_0^3+\sum_{i=1}^{2l} y_i^2,\quad u_{N-1}^0:=F(t_0,\xi_{N-1}^0).
\een 
Let us assume that the critical values $u_i^0:=F(t_0,\xi_i^0)$, $1\leq
i\leq N-1$, are pairwise distinct. Note that $t_0\in B\setminus{B_{ss}}$
and that all other points in $B\setminus{B_{ss}}$ form an analytic
subvariety in $B$ of codimension at least 2. 

Let us choose a small disc $D_i$ with center the critical value
$u_i^0$ for each $i=1,2,\dots, N-1$. We are going to let $t$ vary in a
small open neighborhood $U$ of $t_0$, so that the critical values
$u_i(t)$, $1\leq i\leq N$ of $F(t,\cdot)$  satisfy the conditions
\ben
u_i(t)\in D_i,\quad 1\leq i\leq N-2,\quad u_{N-1}(t),u_N(t)\in D_{N-1},
\een
and 
\ben
u_i(t_0) = u_i^0,\quad 1\leq i\leq N-2,\quad u_{N-1}(t_0)=u_N(t_0) = u_{N-1}^0.
\een
Let us fix an arbitrary $t\in U\cap B_{ss}$. 

\subsection{Twisted representations of the local Heisenberg algebras}
We fix a reference point $p_i$ in the complement $D_i^*$ to the
critical values in $D_i$ and denote by 
\ben
\Delta_i\subset
H_{2l}(X_{t,p_i};\Z), \quad 1\leq i\leq N-1
\een 
the cycles vanishing over
the critical values contained in $D_i$. Note that together with the
intersection pairing $(\cdot |\cdot )$ the sets $\Delta_i=\{\pm
\beta_i\}$, $1\leq i\leq N-2$ are root systems of type $A_1$, while 
\ben
\Delta_{N-1} = \{\alpha:=\beta_{N-1},\beta:=\beta_N,\alpha+\beta,-\alpha,-\beta,-\alpha-\beta\}
\een
is a root system of type $A_2$. 

Let us fix $\Delta=\Delta_i$, $D:=D_i$, and $D^* = D_i^*$ for some
$i=1,2,\dots,N-1$.  We denote by $Q_\Delta$ the corresponding
root lattice and put $\lieh_\Delta=\C\otimes_\Z Q_\Delta$. The vector
space $\widehat{\lieh}_\Delta = \lieh_\Delta[t,t^{-1}] \oplus \C\, K$ has a natural
structure of a Heisenberg Lie algebra with Lie bracket given by 
\ben
[\alpha\,t^m,\beta\,t^n]=m\delta_{m+n,0}\,(\alpha|\beta)\,K\ .
\een
We denote by $\F_\Delta={\rm Sym}(\lieh[t^{-1}]t^{-1})$ the Fock space of
$\widehat{\lieh}_\Delta$, i.e., the unique irreducible highest weight representation
of $\widehat{\lieh}_\Delta $, such that the center $K$ acts by 1 and
$\widehat{\lieh}_\Delta^+:=\lieh_\Delta[t]$ annihilates the {\em vacuum}
$1$. The notation $t$ that appears here has nothing to do with the
deformation parameters that we introduced before. In order to avoid
confusion, from now on we put $a_m:=a\,t^m$, $a\in \lieh_{\Delta},\
m\in \Z$. 
   
Following \cite{BM}, we define {\em bosonic fields}
\beq\label{bos-fields}
X_t(\alpha,\la) = \d_\la\, \widehat{\f_\alpha}(t,\la),\quad 
\eeq
and {\em propagators}
\beq\label{propagator}
P_{\alpha,\beta}(t,\la;\mu-\la ) = \d_\la\d_\mu\, \lim_{\ge\to 0}\
\int_{t-(u_i(t)+\ge){\bf 1}} ^{t-\la\,{\bf 1} } I_\alpha^{(0)}(t',\mu-\la )\bullet I_\beta^{(0)}(t',0) 
\eeq
where $\alpha,\beta \in {\Delta}$, for each $\la \in D^*$ we pick
$\mu\in D^*$ sufficiently close to $\la$, and the
integration is along a path such that $\beta_{t',0}\in
H_{2l}(X_{t',0};\Z)$ vanishes at the end point $t'=t-u_i(t)\one$. The
integrand is a 1-form obtained as follows: each period vector is by
definition a co-vector in $T_{t'}^*B$; we identify vectors and
co-vectors via the residue pairing and hence the Frobenius
multiplication in $T_{t'}B$ induces a multiplication on $T^*_tB$. 
Finally, we can extend the definition of the propagator
bi-linearly to all $\alpha,\beta\in \lieh_\Delta$.

The Laurent series expansion of the propagator \eqref{propagator} at
$\mu=\la$ (see \cite{BM} Lemma 7.5) has the following form:
\ben
P_{\alpha,\beta}(t,\la;\mu-\la) = \frac{(\alpha|\beta)}{(\la-\mu)^2} +
\sum_{k=0}^\infty P_{\alpha,\beta}^k(t,\la)\, (\mu-\la)^k.
\een 
The above series has a non-zero radius of convergence and the
coefficients $P^k_{\alpha,\beta}(t,\la)$ are multi-valued analytic
functions on $D^*$, i.e., the analytic continuation in $\la$ along any
path in $D^*$ is compatible with the monodromy action on $\alpha$ and
$\beta$. The latter statement follows from Lemmas 7.1--7.3 in
\cite{BM}, which can be applied in our settings as well because the root
system $\Delta$ is of type $A$. 

For\/ $a\in\F_\Delta$ of the form
\ben
a = \al^1_{(-k_1-1)} \cdots \al^{r}_{(-k_{r}-1)} 1 \,, \qquad
r\geq1 \,, \; \al^i\in\lieh_\Delta \,, \; k_i \geq0 \,,
\een
we define
\beq\label{wick-1}
X_t(a,\la) = \sum_{J} \, \Bigl( \prod_{ (i,j)\in J } \d_\la^{(k_j)} P^{k_i}_{\al^i,\al^j}(t,\la) \Bigr)
\; {:} \Bigl( \prod_{l\in J'} \d_\la^{(k_l)} X_t(\al^l, \la)  \Bigr)
{:} \, ,
\eeq
where the sum is over all collections\/ $J$ of disjoint ordered pairs\/ $(i_1,j_1),$ $\dots,$ $(i_s,j_s)$ $\subset \{1,\dots,r\}$ such that\/
$i_1<\cdots<i_s$ and\/ $i_l<j_l$ for all $l$,
and\/ $J' = \{1,\dots,r\} \setminus
\{i_1,\dots,i_s,j_1,\dots,j_s\}$. The main property of the above
operators is that their Laurent series expansions at the critical
values contained in $D$ form a twisted representation of $\F_\Delta$
(see \cite{BM} Section 6 for more precise statement). 

\subsection{The local Eynard--Orantin recursion}
Let $\Delta=\Delta_i$ be one of the root systems and let $a =\al^1_{(-k_1-1)} \cdots \al^{r}_{(-k_{r}-1)} 1\in
\F_\Delta$ be any vector. We define a multi-valued analytic symmetric $r$-form 
\ben
\Omega_g^a(t,\la;\t)  = f_g^a(t,\la;\t)\, \underbrace{d\la\cdots d\la}_{r\mbox{\scriptsize\ times}}
\een
as follows
\beq\label{sym-forms}
X_t(a,\la)\, \A_t(\hbar;\q) = \Big(\, \sum_{g=0}^\infty f_g^a(t,\la;\q)\,
\hbar^{g-1}\, \Big) \ \A_t(\hbar;\q).
\eeq
Note that in the definition of $\Omega_g^a(t,\la;\t)$ we replaced $\q$
by $\t$, so we did not use the dilaton-shift identification \eqref{dilaton}. If $a=\al^1_{(-1)} \cdots \al^{r}_{(-1)}
1$; then we will write $\Omega_g^{\alpha^1,\dots,\alpha^r}(t,\la;\t)$
instead of $\Omega_g^a(t,\la;\t)$.

We also need the correlator functions
\beq\label{cor-fun}
\langle v_{i_1}\psi^{k_1},\dots,v_{i_n}\psi^{k_n}\rangle_{g,n}(t;\t)
\eeq
defined by
\ben
\d_{t_{k_1}^{i_1}}\cdots \d_{t_{k_n}^{i_n}} \, \log\ \A_t(\hbar;\t) = 
\sum_{g=0}^\infty
\hbar^{g-1}\,\langle v_{i_1}\psi^{k_1},\dots,v_{i_n}\psi^{k_n}\rangle_{g,n}(t;\t).
\een
Note that 
\ben
\A_t(\hbar;\t) = \exp \Big(\, \sum_{g,n=0}^\infty \,
\frac{\hbar^{g-1}}{n!}\, 
\langle \t(\psi),\dots,\t(\psi)\rangle_{g,n}(t;0)\, \Big),
\een
where by extending multi-linearly the definition \eqref{cor-fun} we
allow the insertions of the correlator to be any formal power series from
$H[\![\psi]\!]$.  

Let us point out that in definition \eqref{sym-forms} 
we are using in an essential way that the total ancestor potential
$\A_t(\hbar;\t)$ is {\em tame}. The latter by definition means that
the correlator functions \eqref{cor-fun} vanish for $\t=0$ and
$k_1+\cdots +k_n>3g-3+n$. The tameness guarantees that inserting formal
power series from $H[\![\psi]\!]$ in the correlators \eqref{cor-fun}
does not produce divergent series. 

The local Eynard--Orantin recursion takes the following form
\beq\label{local-rec}
\langle v_a\,\psi^m\rangle_{g,1}(t;\t) = -\frac{1}{4}\sum_{i=1}^N {\rm
  Res}_{\la=u_i} \,
\frac{\Omega(v_a\,z^m,\f_{\beta_i}^-(t,\la;z))}{y_{\beta_i}(t,\la)}\ \Omega_g^{\beta_i,\beta_i}(t,\la;\t),
\eeq
where $\beta_i$ is a cycle vanishing over $\la=u_i$ and 
\ben
y_\beta (t,\la) := (I^{(-1)}_\beta(t,\la),1)\, d\la .
\een

\subsection{Extending the recursion}
Let $\Delta=\Delta_{N-1}$ be the root system of type $A_2$. Put 
\ben
\chi_1 = \frac{2}{3}\,\alpha +\frac{1}{3}\, \beta,\quad 
\chi_2 = -\frac{1}{3}\,\alpha +\frac{1}{3}\, \beta,\quad 
\chi_3 = -\frac{1}{3}\,\alpha -\frac{2}{3}\, \beta \quad \in \quad
\lieh_\Delta .
\een
We refer to these as 1-point cycles. Note that the root system
$\Delta$ consists of all differences $\chi_i-\chi_j$ for $i\neq j$. 
Motivated by the construction of Bouchard--Eynard \cite{BE} we introduce
the following integral
\beq\label{integral}
-\frac{1}{2\pi\sqrt{-1}}\, \oint \sum_{c_1,\dots,c_r} \,\frac{1}{(r-1)!}\,
\frac{\Omega(v_a\,z^m,\f_{c_1}^-(t,\la;z))}{\prod_{k=2}^r
  y_{c_k-c_1}(t,\la)}\ 
\Omega_g^{c_1,\dots,c_r}(t,\la;\t),
\eeq
where the integral is along a closed loop in $D_{N-1}^*$ that goes
once counterclockwise around the critical values $u_{N-1}(t)$ and
$u_N(t)$ and the sum is over all $r=2,3$ and all $c_1,\dots,c_r\in
\{\chi_1,\chi_2,\chi_3\}$ such that $c_i\neq c_j$ for $i\neq j$. The
monodromy group is the Weyl group of the root system and it acts on
the 1-point cycles via permutations. In other words the integrand is
monodromy invariant, hence a single valued analytic 1-form in
$D_{N-1}^*$, so the integral makes sense.
\begin{theorem}\label{t2}
For $g\neq 1$, 
the integral \eqref{integral} coincides with the sum of the last two
residues (at $\la=u_{N-1},u_N$) in the sum \eqref{local-rec}. For
$g=1$ the same identification holds up to terms independent of $\t$.
\end{theorem}
The proof of Theorem \ref{t2} relies on a certain identity that we
would like to present first. Let $u_i$ and $u_j$
$(1\leq i,j\leq N)$ be two of the critical values, $\be:=\be_j$ be the cycle
vanishing over $u_j$, and $a\in \lieh_{\Delta_i}$ (we assume that $\Delta_N=\Delta_{N-1}$). 
Let us fix some Laurent series
\ben
f(\la,\mu)\in (\la-u_i)^{1/2}\C(\!(\la-u_i,\mu-u_j)\! ) + \C(\!(\la-u_i,\mu-u_j)\! ) 
\een
where $ \C(\!(\la-u_i,\mu-u_j)\! )$ denotes the space of formal
Laurent series. We have to evaluate residues of the following form:
\beq\label{residue-ij}
{\rm Res}_{\la=u_i}\, {\rm Res}_{\mu=u_j}\,\sum_{\scriptsize{\rm all \ branches}}\, 
\frac{\Omega(\phi_a^+(t,\la;z),\f_{\beta}^-(t,\mu;z) )}{y_\be(t,\mu)}\
f(\la,\mu)  \,d\mu ,
\eeq
where $\phi_a$ is the formal series \eqref{phi-alpha} and the sum is
over all branches (2 of them) of the multivalued function that follows. 
\begin{lemma}\label{L1}
If $f(\la,\mu)$ does not have a pole at $\la=u_i$; then the residue
\eqref{residue-ij} is non-zero only if $i=j$ and in the latter case it
equals to
\ben
(a|\be) \  {\rm Res}_{\la=u_i}\sum_{\scriptsize{\rm all \ branches}}
  \,
\frac{f(\la,\la)}{y_\be(t,\la)}\, d\la^2.
\een
\end{lemma}
\proof
Put $a=a'+(a|\be_i)\be_i/2$; then $a'$ is invariant with respect to
the monodromy around $\la=u_i$. From this we get that
$\phi_{a'}^+(t,\la;z)$ is analytic at $\la=u_i$, so it does not
contribute to the residue. In other words, it is enough to
prove the lemma only for $a=\be_i.$ Let us assume that $a=\be_i$.
We have
\ben
\Omega(\phi_a^+(t,\la;z),\f_{\beta}^-(t,\mu;z) )
=\Omega(\f_{\beta}^+(t,\mu;z),\phi_a^-(t,\la;z) ) + 
\Omega(\phi_a(t,\la;z),\f_{\beta}(t,\mu;z) ).
\een
The first symplectic pairing on the RHS does not contribute to the
residue, because $\phi_a^-(t,\la;z) $ has a pole of order at most
$\frac{1}{2}$ so after taking the sum over all branches, the poles of
fractional degrees cancel out and hence the 1-form at hands is analytic at
$\la=u_i$. For the second symplectic pairing we recall Lemma
\ref{vanishing_a1} and after a straightforward computation we get
\ben
\Omega(\phi_{A_1}(u_i,\la;z)e_i,\f_{A_1}(u_j,\mu;z)e_j) =
2\delta_{i,j}\, \frac{(\mu-u_j)^{\frac{1}{2}}
}{(\la-u_i)^{\frac{1}{2}} } \, \delta(\la-u_i,\mu-u_j)\, d\la ,
\een
where
\ben
\delta(x,y)=\sum_{n\in \Z}\, x^n y^{-n-1}
\een
is the formal $\delta$-function. It is an easy exercise to check that
for every $f(y)\in \C(\!(y)\! )$ we have
\ben
{\rm Res}_{y=0}\  \delta(x,y) \, f(y) = f(x). 
\een
The lemma follows.
\qed
\subsection{Proof of Theorem \ref{t2}}
The integral \eqref{integral} can be written as a sum of two residues:
${\rm Res}_{\la=u_{N-1}}$ and ${\rm Res}_{\la=u_{N}}$. We claim that each
of these residues can be reduced to the corresponding residue in the
sum \eqref{local-rec}. Let us present the argument for
$\la=u_{N-1}$. The other case is completely analogous. 

Recall that we denoted by $\alpha=\be_{N-1}$ the cycle vanishing over
$u_{N-1}$. The summands in \eqref{integral} for which $r=2$ and $c_1,c_2\in
\{\chi_1,\chi_2\}$ give precisely 
\ben
{\rm Res}_{\la=u_{N-1}}
\frac{\Omega(v_a\,z^m,\f_{\chi_1-\chi_2}^-(t,\la;z))}{
  y_{\chi_1-\chi_2}(t,\la)}\ 
\Omega_g^{\chi_1,\chi_2}(t,\la;\t).
\een
On the other hand, using that $\alpha=\chi_1-\chi_2$
we get
\ben
\Omega_g^{\chi_1,\chi_2}(t,\la;\t) = -\frac{1}{4} \,
\Omega_g^{\alpha,\alpha}(t,\la;\t) + \frac{1}{4}\, \Omega_g^{\chi_1+\chi_2,\chi_1+\chi_2}(t,\la;\t)
\een
Since $(\chi_1+\chi_2|\alpha) = 0$, the form
$\Omega_g^{\chi_1+\chi_2,\chi_1+\chi_2}(t,\la;\t)$ is analytic at
$\la=u_{N-1}$, so it does not contribute to the residue. Therefore we
obtain precisely the $(N-1)$-st residue in 
\eqref{local-rec}. It remain only to see that the remaining summands
with $r=2$ cancel out with the summand with $r=3$. 

There are two types of quadratic summands: $c_1,c_2\in
\{\chi_1,\chi_3\}$ and $c_1,c_2\in
\{\chi_2,\chi_3\}$. They add up respectively to 
\beq\label{Q1}
\frac{\Omega(v_a\,z^m,\f_{\chi_1-\chi_3}^-(t,\la;z))}{
  y_{\chi_1-\chi_3}(t,\la)}\ 
\Omega_g^{\chi_1,\chi_3}(t,\la;\t)
\eeq
and 
\beq\label{Q2}
\frac{\Omega(v_a\,z^m,\f_{\chi_2-\chi_3}^-(t,\la;z))}{
  y_{\chi_2-\chi_3}(t,\la)}\ 
\Omega_g^{\chi_2,\chi_3}(t,\la;\t).
\eeq
By definition
\beq\label{2pt-cor}
\sum_{g=0}^\infty \hbar^{g-1}\,\Omega^{\chi_i,\chi_3}_g(t,\la;\t)\A_t = \Big(
:\widehat{\phi}_{\chi_i}(t,\la) \widehat{\phi}_{\chi_3}(t,\la) : +
P^0_{\chi_i,\chi_3}(t,\la) \, d\la^2 \Big)\ \A_t.
\eeq
The term $P^0_{\chi_i,\chi_3}$ contributes only to genus 1 and the
contribution is independent of $\t$, so we may ignore this term.  
The normal product on the RHS is by definition 
\beq\label{np}
\widehat{\phi}_{\chi_3}(t,\la)\, \widehat{\phi}_{\chi_i}^+(t,\la) + 
\widehat{\phi}_{\chi_i}^-(t,\la)\widehat{\phi}_{\chi_3}(t,\la).
\eeq
Since $(\chi_3|\alpha)=0$ the field $\widehat{\phi}_{\chi_3}(t,\la)$
is analytic at $\la=u_i$. In addition
$\widehat{\phi}_{\chi_i}^-(t,\la)$ has a pole of order at most
$\frac{1}{2}$ at $\la=u_i$. It follows that the second summand in
\eqref{np} does not contribute to the residue and therefore it can be
ignored as well. For the RHS of \eqref{2pt-cor} we get 
\ben
\sum_{g=0}^\infty 
\hbar^{g-1}\, \widehat{\phi}_{\chi_3}(t,\la)\,\langle
\phi_{\chi_i}^+(t,\la;\psi)\rangle_{g,1}(t;\t)\, \A_t.
\een
Recalling the local recursion \eqref{local-rec} we get
\ben
-\frac{1}{4}\,\sum_{j=1}^N\ {\rm Res}_{\mu=u_j}\
\frac{\Omega(\phi_{\chi_i}^+(t,\la;z),\f_{\beta_j}^-(t,\mu;z) )
}{y_{\beta_j}(t,\mu) } \
\widehat{\phi}_{\chi_3}(t,\la)\,Y(\beta_j^2,\mu)\,d\mu^2\, \A_t.
\een
Therefore we need to compute the residues 
${\rm Res}_{\la=u_{N-1}}\ {\rm Res}_{\mu=u_j}$ of the following expressions
\ben
-\frac{1}{4}\,\sum_{i=1,2} \ 
\frac{\Omega(v_a\,z^m,\f_{\chi_i-\chi_3}^-(t,\la;z))}{
  y_{\chi_i-\chi_3}(t,\la)}\ 
\frac{\Omega(\phi_{\chi_i}^+(t,\la;z),\f_{\beta_j}^-(t,\mu;z) )
}{y_{\beta_j}(t,\mu) } \
\widehat{\phi}_{\chi_3}(t,\la)\,Y(\beta_j^2,\mu)\, \,d\mu^2\, \A_t.
\een
The operator 
$\widehat{\phi}_{\chi_3}(t,\la)\,Y(\beta_j^2,\mu) \,d\mu^2\, $ can be
written as 
\beq\label{cubic-op}
:\widehat{\phi}_{\be_j}(t,\mu)^2　\,
\widehat{\phi}_{\chi_3}(t,\la)　: +
2[\widehat{\phi}^+_{\chi_3}(t,\la),
\widehat{\phi}^-_{\be_j}(t,\mu)]\ \widehat{\phi}_{\be_j}(t,\mu)　+ 
P_{\be_j,\be_j}^0(t,\mu)\widehat{\phi}_{\chi_3}(t,\la) \,d\mu^2\, .　　
\eeq
Since $(\chi_3|\alpha)=0$ the operator
$\widehat{\phi}^+_{\chi_3}(t,\la)$ is regular at $\la=u_i$. It follows
that the commutator
\ben
[\widehat{\phi}^+_{\chi_3}(t,\la),
\widehat{\phi}^-_{\be_j}(t,\mu)]\in \C(\!( \la-u_{N-1}, \mu-u_j)\! )
\een
and therefore we may
recall Lemma \ref{L1}. The above residue is non-zero only if
$j=N-1$. In the latter case we get
\beq\label{residue}
-\frac{1}{4}\,{\rm Res}_{\la=u_{N-1}}\ \sum_{i=1,2} \ 
(\chi_i|\al)\ 
\frac{\Omega(v_a\,z^m,\f_{\chi_i-\chi_3}^-(t,\la;z))}{
  y_{\chi_i-\chi_3}(t,\la)\, y_{\al}(t,\la) }\ 
\widehat{\phi}_{\chi_3}(t,\la)\,Y(\al_{-1}^2,\la)\, \A_t.
\eeq
Note that (c.f. \cite{BM}, Section 7)
\ben
[\widehat{\phi}^+_{\chi_3}(t,\la),
\widehat{\phi}^-_{\be_j}(t,\mu)] = 
\iota_{\la-u_{N-1}}\,\iota_{\mu-u_{N-1}}\
P_{\chi_3,\be_j}(t,\la;\mu-\la), 
\een
where $\iota_{\la-u_{N-1}}$ is the Laurent series expansion at
$\la=u_{N-1}$. Hence
\ben
\widehat{\phi}_{\chi_3}(t,\la)\,Y(\al_{-1}^2,\la) =
\iota_{\la-u_{N-1}}
X_t((\chi_3)_{-1}\al_{-1}^2,\la).
\een
By definition 
\ben
-\frac{1}{4}\, \al_{-1}^2= (\chi_1)_{-1}\,(\chi_2)_{-1} -\frac{1}{4}\,(\chi_3)_{-1}^2
\een
and since $\chi_3$ is invariant with respect to the local monodromy
around $\la=u_{N-1}$, the field $X_t ((\chi_3)_{-1}^3,\la)$ does not
contribute to the residue. We get the following formula for the
residue \eqref{residue}:
\ben
{\rm Res}_{\la=u_{N-1}}\ \sum_{i=1,2} \ 
(\chi_i|\al)\ 
\frac{\Omega(v_a\,z^m,\f_{\chi_i-\chi_3}^-(t,\la;z))}{
  y_{\chi_i-\chi_3}(t,\la)\, y_{\al}(t,\la) }\ 
Y((\chi_1)_{-1}\,(\chi_2)_{-1}\,(\chi_3)_{-1},\la)\, \A_t.
\een
Using that $\al=\chi_1-\chi_2$, $(\chi_1|\al) = 1$, and $(\chi_2|\al)=-1$
we get
\ben
& {\rm Res}_{\la=u_{N-1}} &  
\left(
\frac{\Omega(v_a\,z^m,\f_{\chi_1}^-(t,\la;z))}{
  y_{\chi_2-\chi_1}(t,\la)\, y_{\chi_3-\chi_1}(t,\la) }\ +
 \frac{\Omega(v_a\,z^m,\f_{\chi_2}^-(t,\la;z))}{
  y_{\chi_1-\chi_2}(t,\la)\, y_{\chi_3-\chi_2}(t,\la) }\ +\right.
\\
&&
\left.+\frac{\Omega(v_a\,z^m,\f_{\chi_3}^-(t,\la;z))}{
  y_{\chi_1-\chi_3}(t,\la)\, y_{\chi_2-\chi_3}(t,\la) }
\right) \times
\sum_{g=0}^\infty \, \hbar^{g-1}
\Omega_g^{\chi_1,\chi_2,\chi_3}(t,\la;\t)\, \A_t.
\een
This sum cancels out the contribution to the residue at $\la=u_{N-1}$ of the cubic terms (i.e. the
terms with $r=3$) of the integral \eqref{integral}.
\qed

\subsection{Proof of Theorem \ref{t1}}
The most difficult part of the proof is already completed. We just
need to take care of several initial cases. Let $t_0$ be a generic
point in $B\setminus{ B_{ss} }$. 
Let us write each 1-point correlator as a sum of terms homogeneous in
$\t$
\ben
\langle v_a\psi^m\rangle_{g,1}(t;\t) =\sum_{n=0}^\infty \frac{1}{n!}
\, \langle v_a\psi^m,\t(\psi),\dots,\t(\psi)\rangle_{g,n+1}(t;0).
\een
We claim that each summand is analytic at $t=t_0.$ In order to see
this let us put a lexicographical order on the summands according to
$(g,n)$ -- genus and degree. Note that the local Eynard--Orantin recursion
tells us how to find a correlator of a fixed genus and degree in terms
of correlators of lower lexicographical order. Assuming that the lower
order correlators are analytic at $t=t_0$ and that we can apply
Theorem \ref{t2}, we get the analyticity of the next correlator,
because the integral \eqref{integral} is analytic at $t=t_0$. 
Therefore, the proof would be completed by induction if we establish
the initial cases
\ben
\langle v_a\psi^m,\t(\psi),\t(\psi)\rangle_{0,3}(t;0)\quad
\mbox{and}\quad 
\langle v_a\psi^m\rangle_{1,1}(t;0).
\een
Note that the above genus-1 correlators are the only ones for which
Theorem \ref{t2} can not be applied. Hence if we verify the
analyticity of the above correlators; then the proof of the
analyticity of the 1-point correlators will be completed. 

We will compute the above correlators via the local recursion. 
Using Lemma \ref{vanishing_a1} we can express the Laurent series
expansion 
\ben
\iota_{\la-u_i}\,\iota_{\mu-u_i}\ 
P_{\be_i,\be_i}(t,\la;\mu-\la)
\een
in terms of the Givental's higher-genus reconstruction operator
$R$. After a straightforward computation we get
\ben
\frac{\mu+\la-2u_i}{(\la-\mu)^2 (\la-u_i)^{1/2}(\mu-u_i)^{1/2}} + \sum_{k,l=0}^\infty 2^{k+l+1}
(e_i,V_{kl}e_i)\, \frac{(\mu-u_i)^{k-\frac{1}{2}}}{(2k-1)!!}\,
\frac{(\la-u_i)^{l -\frac{1}{2} }}{(2l-1)!!},
\een
where the matrices $V_{kl}\in {\rm End}(\C^N)$ are defined as follows:
\ben
\sum_{k,l=0}^\infty V_{kl}\, w^kz^l = \frac{1-\leftexp{T}{R}(-w)R(-z)}{w+z}.
\een
The Laurent series expansion of the propagator at $\mu=\la$ becomes
\ben
\frac{2}{(\la-\mu)^2} + P^0_{\be_i,\be_i}(t,\la)+\cdots,
\een
where the dots stand for higher order terms in $(\mu-\la)$ and 
\ben
 P^0_{\be_i,\be_i}(t,\la) = \frac{1}{4} (\la-u_i)^{-2} + 2(e_i,R_1
 e_i)\, (\la-u_i)^{-1}. 
\een
Similarly, we can find the Laurent series expansion of the period
vectors
\ben
I^{(-1-m)}_{\be_i}(t,\la) =
2\frac{(2(\la-u_i))^{m+\frac{1}{2} }}{(2m+1)!!} \Big( e_i -\sum_{k=1}^N
  \frac{R_1^{ki}  }{2m+3}\, e_k\,2(\la-u_i) + \cdots \Big), 
\een
where $R_1^{ki}$ is the $(k,i)$-th entry of the matrix $R_1$ and
slightly abusing the notation we put $e_i=du_i/\sqrt{\Delta_i}$. 

Let us begin with the genus-0 case. The quadratic part of the
form $\Omega_0^{\be_i,\be_i}(\la;\t)$ is 
\ben
\sum_{k,l=0}^\infty \sum_{a,b=1}^N \, (I^{(-k)}_{\be_i}(t,\la),v_a) \,
(I^{(-l)}_{\be_i}(t,\la),v_b)\, d\la^2\, t_k^a\, t_l^b.  
\een
Applying the local recursion and leaving out the terms with $k>0$ or
$l>0$ (since they do not contribute to the residue) we get that $\langle
v_c\psi^m,\t(\psi),\t(\psi)\rangle_{0,3}(t;0) $ is
\ben
\frac{1}{4}\sum_{i=1}^N {\rm Res}_{\la=u_i} \, \sum_{a,b=1}^N\, 
\frac{(I^{(-1-m)}_{\be_i}(t,\la),v_c) }{
  (I^{(-1)}_{\be_i}(t,\la),1)  } \,  (I^{(0)}_{\be_i}(t,\la),v_a) \,
(I^{(0)}_{\be_i}(t,\la),v_b)\, d\la\, t_0^a\, t_0^b
\een
The above residue is non-zero only for $m=0$. Using the Laurent series
expansion of the periods we get 
\ben
\sum_{a,b=1}^N\, \frac{1}{2}\, \sum_{i=1}^N \frac{1}{\Delta_i} \, 
\frac{\d u_i}{\d\tau_a} \frac{\d u_i}{\d\tau_b} \frac{\d
  u_i}{\d\tau_c} \, t_0^a\, t_0^b =\sum_{a,b=1}^N\,  \frac{1}{2}\,
(v_a\bullet v_b,v_c) \, t_0^a\, t_0^b.
\een
The coefficients of the above quadratic form are precisely the
structure constants of the Frobenius multiplication, so they are
analytic. 

Let us continue with the genus-1 case. Now the local recursion takes
the form
\ben
\langle
v_a\psi^m
\rangle_{1,1}(t;0) = \frac{1}{4} \sum_{i=1}^N {\rm Res}_{\la=u_i} \,  
\frac{(I^{(-1-m)}_{\be_i}(t,\la),v_c) }{
  (I^{(-1)}_{\be_i}(t,\la),1)  } \, P_{\be_i,\be_i}^0(t,\la). 
\een
The above residue is non-zero only if $m=0$ or $m=1$, because
$P^0_{\be_i,\be_i}$ has a pole of order at most 2. In the case when
$m=1$, after substituting the Laurent series expansions of the
propagator and of the periods we get
\ben
\langle
v_a\psi
\rangle_{1,1}(t;0) = 
\frac{1}{24}\,\sum_{i=1}^N \frac{\d u_i}{\d \tau_a} =
\frac{1}{24}\,\frac{\d}{\d \tau_a} \ {\rm Tr}(E\bullet)
=\frac{1}{24}\,{\rm Tr}(v_a\bullet) ,
\een
where we used that in canonical coordinates the Euler vector field
takes the form $\sum_{i=1}^N u_i\d_{u_i}$ and hence $\sum_i u_i = {\rm
Tr}(E\bullet)$. 

For $m=0$ after a straightforward computation, using also that
$R_1^{ki}=R_1^{ik}$, we get
\ben
\langle
v_a
\rangle_{1,1}(t;0) = 
\frac{1}{2} \sum_{i=1}^N R_1^{ii} \frac{\d u_i}{\d \tau_a} + 
\frac{1}{24} \sum_{k,i=1}^N \frac{\sqrt{\Delta_k}}{\sqrt{\Delta_i}}\,
R_1^{ki}\, \Big(\  \frac{\d u_k}{\d \tau_a} - \frac{\d u_i}{\d \tau_a}
\ \Big).
\een
The differential equations \eqref{de_1}--\eqref{de_2} imply the
following relation:
\ben
[\d_a U_t,R_1]=\Psi_t^{-1}\,\d_a\Psi_t.
\een
From this equation, using that the entries of the matrix $\Psi$ and
$\Psi^{-1}$ are respectively 
\ben
\Psi^{bi} = \sqrt{\Delta_i}\, \frac{\d \tau_b}{\d u_i}\quad\mbox{and}\quad
(\Psi^{-1})^{kb} = \frac{1}{\sqrt{\Delta_k}} \, \frac{\d u_k}{\d
  \tau_b} 
\een
we get 
\ben
\frac{\sqrt{\Delta_k}}{\sqrt{\Delta_i}}\,
R_1^{ki}\, \Big(\  \frac{\d u_k}{\d \tau_a} - \frac{\d u_i}{\d
  \tau_a}\ \Big) = \frac{\d u_k}{\d
  \tau_b} \d_a\Big(\sqrt{\Delta_i}\, \frac{\d \tau_b}{\d u_i}\Big)
\,\frac{1}{\sqrt{\Delta_i}} =\delta_{ki}\,\frac{1}{2}\, \d_a\,\log
\Delta_i + \frac{\d u_k}{\d
  \tau_b} \frac{\d}{\d \tau_a}\Big(\  \frac{\d \tau_b}{\d u_i}\ \Big).
\een
If we sum the above expression over all $i=1,2,\dots,N$, since $\sum_i
\d_{u_i} = \d_N$, we get simply $\frac{1}{2}\d_a \log \Delta_k$. Hence
the 1-point genus-1 correlator becomes
\ben
\langle
v_a
\rangle_{1,1}(t;0) = 
\frac{1}{2} \sum_{i=1}^N R_1^{ii} \frac{\d u_i}{\d \tau_a} + 
\frac{1}{48} \sum_{k=1}^N \d_a\,\log \Delta_k.
\een
The RHS is a well known expression, i.e., it is $\d_a F^{(1)}(t)$,
where $F^{(1)}(t)$ is the genus-1 potential of the Frobenius
structure (see \cite{G4}), also known as the $G$-{\em
  function} (see \cite{DZ,DZ2}). According to Hertling \cite{He}, Theorem 14.6, the
function $F^{(1)}(t)$ is analytic. This completes the proof of the
analyticity of all correlators that have at least 1 insertion. 

To finish the proof of Theorem \ref{t2} we still have to prove that 
the correlators with no insertions $\langle \, \rangle_{g,0}(t;0)$ are
analytic. Such correlators are identically $0$ for $g=0$,
due to the tameness property of the ancestor potential. In genus 1, in
the settings of Gromov--Witten theory, the correlator is 0 because the
moduli space $\overline{\M}_{1,0}$ is empty. It is not hard to check
(using the differential equation \eqref{de_1})
that in the abstract settings of semis-simple Frobenius manifolds this
correlator still vanishes. For higher
genera, using the differential equation \eqref{de_1} one can check
easily that 
\ben
\d_a \, \langle \, \rangle_{g,0}(t;\t) = \frac{\d}{\d t_0^a}\, \langle
\, \rangle_{g,0}(t;\t) = \langle v_a\rangle_{g,1}(t;\t).
\een
In other words, the differential of the correlator $\langle \, \rangle_{g,0}(t;0)$ is
an analytic 1-form on $B$ and since $B$ is simply connected the
correlator must be analytic as well.
\qed

\section*{Acknowledgements} 
I am thankful to B. Bakalov and Y. Ruan for many stimulating discussions.
This work is supported by Grant-In-Aid and by the World
Premier International Research Center Initiative (WPI Initiative),
MEXT, Japan.


\begin{thebibliography}{JKV2}

\bibitem{AGV}{V.~Arnold, S.~Gusein-Zade, and A.~Varchenko.}
{\em Singularities of Differentiable maps.} Vol. II. Monodromy and
  Asymptotics of Integrals. Boston, MA: Birkh\"auser Boston,
  1988. viii+492 pp

\bibitem{BM}{B.~Bakalov and T.~Milanov}
\emph{W-constraints for the total descendant potential of a simple singularity.}
Compositio Math., doi: 10.1112/S0010437X12000668, (2012), 1--49
http:/\!/arxiv.org/abs/1203.3414

\bibitem{BE}{V.~Bouchard and B.~Eynard}
\emph{Think globally, compute locally.}
Preprint (2012);
http:/\!/arxiv.org/abs/1211.2302


\bibitem{BOSS}{P.~Dunnin--Barkowski, N.~Orantin, S.~Shadrin, and  L.~Spitz.}
{\em Identification of the Givental formula with the spectral curve topological recursion procedure.} 
Preprint (2012);
http:/\!/arxiv.org/abs/1211.4021

\bibitem{BPS}
A.~Buryak, H.~Posthuma, and S.~Shadrin,
\textit{On deformations of quasi-Miura transformations and the
  Dubrovin--Zhang bracket}.
J. Geom. Phys. \textbf{62}(2012), no. 7, 1639--1651 
http:/\!/arxiv.org/abs/1104.2722

\bibitem{EO}{B.~Eynard and N.~Orantin,}
\emph{Invariants of algebraic curves and topological expansion.} 
Comm. in Number Theory and Physics \textbf{1}(2007), 347--552

\bibitem{Du}
B.~Dubrovin,
\textit{Geometry of 2D topological field theories}. 
In: ``Integrable systems and quantum groups'' 
(Montecatini Terme, 1993), 120--348, Lecture Notes
in Math., 1620, Springer, Berlin, 1996

\bibitem{DZ}
B.~Dubrovin and Y.~Zhang,
\textit{Bi-Hamiltonian hierarchies in $2$D topological field theory at one-loop approximation}.
Comm. Math. Phys. \textbf{198} (1998), 311--361

\bibitem{DZ2}
B.~Dubrovin and Y.~Zhang,
\textit{Frobenius Manifolds and Virasoro constraints.}
Sel. Math., New ser. {\bf 5}(1999), 423--466

\bibitem{G1} A.~Givental.
{\em Semisimple Frobenius structures at higher genus.}
Internat. Math. Res. Notices 2001, no. 23, 1265-1286

\bibitem{G2}
A.~Givental.
\textit{Gromov--Witten invariants and quantization of quadratic Hamiltonians}. 
Mosc. Math. J. \textbf{1} (2001), 551--568

\bibitem{G3}{A.~Givental.}
\emph{$A_{n-1}$ singularities and $n$KdV Hierarchies.}
Mosc. Math. J. \textbf{3}(2003), no.2, 475--505

\bibitem{G4}{A.~Givental.}
\emph{Elliptic Gromov--Witten invariants and the generalized mirror
  conjecture.}
Integrable systems and algebraic geometry. Proceeding of the Tanaguchi
Symposium 1997(M.-H. Saito, Y. Shimizu, K. Ueno, eds). World
Scientific, River Edge NJ 1998, 107--155

\bibitem{He}{C.~Hertling.}
\emph{Frobenius Manifolds and Moduli Spaces for Singularities.}
Cambridge Tracts in Mathematics, 151. Cambridge University Press, Cambridge, 2002. x+270 pp

\bibitem{Ko1}
M.~Kontsevich, 
\textit{Intersection theory on the moduli space of curves and the matrix Airy function}. 
Comm. Math. Phys. \textbf{147} (1992), 1--23

\bibitem{M}{Milanov, Todor,}
\emph{
The Eynard--Orantin recursion for the total ancestor potential.}
Preprint(2012); 
http:/\!/arxiv.org/abs/1211.5847

\bibitem{MR}{Milanov, Todor; Ruan, Yongbin,}
\emph{
Gromov-Witten theory of elliptic orbifold $\mathbb{P}^1$ and quasi-modular forms.}
Preprint(2011);
http:/\!/arxiv.org/abs/1106.2321

\bibitem{MRS}{Milanov, Todor; Ruan, Yongbin; Shen, Yefeng,}
\emph{
Gromov--Witten theory and cycled-valued modular forms.}
Preprint(2012);
http:/\!/arxiv.org/abs/1206.3879

\bibitem{Sa}{K.~Saito,}
{\em On Periods of Primitive Integrals, I.}
Preprint RIMS(1982)

\bibitem{SaT}{K.~Saito and A.~Takahashi,}
{\em From primitive forms to Frobenius manifolds.}
From Hodge theory to integrability and TQFT tt*-geometry, 31-48, Proc. Sympos. Pure Math., 78, Amer. Math. Soc., Providence, RI, 2008

\bibitem{MS} 
M.~Saito, 
\textit{On the structure of Brieskorn lattice}. 
Ann. Inst. Fourier \textbf{39} (1989), 27--72

\bibitem{Te}
C.~Teleman, 
\textit{The structure of 2D semi-simple field theories}.
Invent. Math. \textbf{188}(2012), no.3, 525--588
http:/\!/arxiv.org/abs/0712.0160 

\bibitem{W1}
E.~Witten,
\textit{Two-dimensional gravity and intersection theory on moduli space}. 
In: ``Surveys in differential geometry,'' 243--310, 
Lehigh Univ., Bethlehem, PA, 1991

\end{thebibliography}
\end{document}